\documentclass{amsart}
\usepackage{verbatim,amssymb,amsmath,amscd,latexsym,amsbsy,mathrsfs}
\usepackage{enumerate}
 \input{xy}
\xyoption{all}

\newtheorem{thm}{Theorem}[section] \newtheorem{pro}[thm]{Proposition}
\newtheorem{lemma}[thm]{Lemma}

\numberwithin{equation}{section}

\theoremstyle{remark}

\theoremstyle{definition} 
\newtheorem{rmk}[thm]{Remark}

\DeclareMathAlphabet{\mathpzc}{OT1}{pzc}{m}{it}

 \DeclareMathOperator*{\QF}{QF}
\DeclareMathOperator*{\Gal}{Gal}

\DeclareMathOperator*{\Char}{char}

 \newcommand{\NN}{\mathbb{N}}
 
 \newcommand{\Aff}{\mathbb{A}}
\newcommand{\PP}{\mathbb{P}}

\newcommand{\cO}{\mathcal{O}}

\newcommand{\Z}{\mathbb{Z}}

\begin{document}
\title{Galois closure of henselization} 
 \author{
   Manish Kumar
  }
  \address{
Statistics and Mathematics Unit\\
Indian Statistical Institute, \\
Bangalore, India-560059
  }
  \email{manish@isibang.ac.in}
 \begin{abstract}
  It is shown that the Galois closure of the henselization of a one dimensional local field arising in geometric and arithmetic situation is separably closed.
 \end{abstract}
\maketitle

\section{Introduction}

Let $k$ be an algebraically closed field and $C$ be a regular $k$-curve. Let $\tau$ be a closed point of $C$ and $R=\cO_{C,\tau}$. Let $R^h$ denote the henselization of $R$ and $\hat R$ the completion of the DVR $R$. Recall that $R^h$ is the integral closure of $R$ in $\hat R$. Let $K$ be the function field $k(C)$ and $K^h=\QF(R^h)$ be the fraction field of $R^h$. One of the main result proved here is the following:

\begin{thm} \label{mainthm}
The Galois closure of $K^h/K$ is the separable closure of $k(C)$. 
\end{thm}

The proof when the base field $k$ has characteristic 0 is fairly simple. In the case when $\Char(k)=p>0$, the proof uses the existence of Harbater-Katz-Gabber covers. Recall that given any $I$-Galois extension $L/k((x^{-1}))$ it extends uniquely to an $I$-cover $Y\to\PP^1_x$ which is totally ramified at most at $x=0$ and $x=\infty$ such that the cover is tamely ramified at $x=0$ and the local field extension at $x=\infty$ agrees with the extension $L/k((x^{-1}))$ [\cite{KatzGabber}, \cite{harbater-adding.branch.points}]. Such a cover will be called the Harbater-Katz-Gabber cover associated to $L/k((x^{-1}))$.

An arithmetic analogue of the above result is also proved in section \ref{arithmetic-case}. More precisely, it is shown that the Galois closure of the strict henselization of a number field at a prime is algebraically closed. Finally in section \ref{HKG} an analogue of the existence of Harbater-Katz-Gabber cover in the number field case is explored.

\section{Lemmas}

We begin with the following lemma which reduces the problem involving infinite extensions of fields to certain statements about finite extensions of fields.
\begin{lemma}\label{red.to.fin.extn}
Let $K\subset M$ be fields. Let $\bar M$ be the separable closure of $M$. Let $K^a$ and $\bar K$ be the separable closure of $K$ in $M$ and $\bar M$ respectively. 
Given a finite separable extension $L_1/M$ with $L_1\subset \bar M$, suppose there exists a finite separable extension $L/K$ with Galois closure $\tilde L$ such that $L\subset M$ and the compositum $\tilde L M$ contains $L_1$. Then the Galois closure of $K^a/K$ is $\bar K$.
\end{lemma} 

\begin{proof}
 Since $\bar M$ is separably closed, so is $\bar K$. Since the Galois closure of $K^a/K$ is contained in $\bar K$, it is enough to show that $\bar K$ is contained in the Galois closure of $K^a/K$. Let $\alpha\in\bar K$ then $M(\alpha)/M$ is a finite separable extension. So by the hypothesis, there exists a finite separable extension $L/K$ with Galois closure $\tilde L$ such that $L\subset M$ and $\tilde LM$ contains $\alpha$. But $\alpha$ is separable over $K$ and separable closure of $K$ in $\tilde LM$ is $\tilde LK^a$, so $\alpha \in \tilde LK^a$. But $\tilde L K^a$ is obviously contained in the Galois closure of $K^a/K$, since $L\subset K^a$. 
\end{proof}

Let $I$ be a finite group and $l$ be positive integer coprime to $|I|$. Let $G$ be the wreath product $\Z/l\Z \wr I=(\Z/l\Z)^I\rtimes I$ and $q:G\to I$ the canonical projection.

\begin{lemma}\label{groupconstruction}
 Let $I_1\le G$ be a subgroup such that $I_1$ is isomorphic to $I$. Then $\cap_{g\in G}gI_1g^{-1}=\{e\}$. In other words, the maximal subgroup of $I_1$ which is also a normal subgroup of $G$ is $\{e\}$.
\end{lemma}

\begin{proof}
 If $I=\{e\}$ then there is nothing to prove. The normal subgroup $(\Z/l\Z)^I$ has order prime to $|I|$. Hence $q(I_1)=I$. Let $(f,i)\in \cap_{g\in G}gI_1g^{-1}$ for some $i\ne e$. Let $i_1\in I$ be different from $e$, then there exists a unique $e \ne i_0\in I$ such that $i=i_1i_0i_1^{-1}$. Moreover, $q|_{I_1}$ is an isomorphism onto $I$ implies that there exists $f_0\in (Z/l\Z)^I$ such that $(f_0,i_0)\in I_1$. By assumption $(f,i)\in \cap_{f_1\in (\Z/l\Z)^I}(f_1,i_1)I_1(f_1,i_1)^{-1}$. So for all $f_1\in (\Z/l\Z)^I$, we have $(f_1,i_1)(f_0,i_0)(f_1,i_1)^{-1}=(f,i)$. But this is clearly false.
\end{proof}

\section{Characteristic zero case}
First we shall give a proof of Theorem \ref{mainthm} when $\Char(k)=0$. Let $x$ be a regular parameter of $R$. Then $\hat R=k[[x]]$ and its fraction field $\hat K=\QF(\hat R)=k((x))$. It is well known (Newton's theorem) that the algebraic closure of $\hat K$ is $\Omega=k((x))[x^{1/n}; n\in \NN]$. We shall view $K^h$, $R^h$, etc. as subrings of $\Omega$.

In view of Lemma \ref{red.to.fin.extn}, it is enough to show that for every $n\in \NN$, there exists a finite extension $L/K$ with Galois closure $\tilde L$ such that $L\subset \hat K$ and $n$ divides $[\tilde L\hat K:\hat K]$.

Note that the polynomial $f(Y)=Y^{n+1}- Y^n-x$ is irreducible in $k[x][Y]$ and hence in $k(x)[Y]$. Let $C$ be the normalization of $\Aff^1_x$ in $k(x)[Y]/(f)$. There are two points lying above $x=0$ $P_0=(x=0,Y=0)$ and $P_1=(x=0,Y=1)$ in the plane curve $f(Y)=0$. Also note that this curve is regular at the two points $P_0$ and $P_1$. So these are the only two points in the normalization $C$ lying above $x=0$.

Note that $f(Y)\cong Y^n(Y-1) (\mod x)$. So by Hensel's lemma there exists $\alpha \in k[[x]]$ such that $f(\alpha)=0$. Let $L=k(x)(\alpha)$ and $\tilde L$ be the Galois closure of $L/k(x)$. Note that $L$ is the function field of $k(C)$. Since there are exactly two points above $x=0$ and $P_1$ is unramified for the cover $C\to \Aff^1_x$, the ramification index at the point $P_0$ is $n$. Hence the inertia group of $\tilde L$ above a point at $x=0$ is a multiple of $n$. This exactly means that $n$ divides $[\tilde L\hat K:\hat K]$.

\section{Positive characteristic case}
When the $\Char(k)=p\ne 0$ then the situation is a little more complicated. The algebraic closure of $k((x))$ is a complicated field. In fact, the absolute Galois group of $k((x))$ is a profinite group of the rank same as the cardinality of $k$.

Let $G$ be a finite group and $\Phi:X\to \PP^1_x$ be a $G$-Galois cover. We shall view $k(X)$ as a subfield of $\Omega=\overline{k((x))}$. Let $y_1,\ldots, y_r$ be points in $X$ lying above $x=0$. The inclusion of $k(X)$ into $\Omega$ corresponds to choosing a point of $X$ above $x=0$ (the center of the unique valuation of $k(X)k((x))$ on $X$). Let $y=y_1$ be this point. Let $I\le G$ be the inertia group of $\Phi$ at $y$. Let $I'\le I$ be the maximal subgroup of $I$ which is a normal subgroup of $G$.

\begin{lemma}\label{galoiscorr}
 If $I'=\{e\}$ then the Galois closure of $k(X)^I/k(x)$ is $k(X)$.
\end{lemma}
\begin{proof}
 By definition of $I'$ and the Galois correspondence $k(X)^{I'}$ is the Galois closure of the field extension $k(X)^I/k(x)$.
\end{proof}

\begin{lemma}\label{splittingramification}
 The field $k(X)^I$ is a subfield of $k((x))$.
\end{lemma}
\begin{proof}
 By \cite[Corollary 4, Section 8.5, Chapter 6]{Bourbaki} $\Gal(k(X)k((x))/k((x)))$ is the inertia group of $\Phi$ at $y$. Hence $I=\Gal(k(X)k((x))/k((x)))$. Hence the fixed field $k(X)^I$ is contained in $k((x))$.
\end{proof}

\begin{proof}[Proof of Theorem \ref{mainthm} in positive characteristic case]
 Let $\hat L/k((x))$ be a finite Galois extension with Galois group $I$. Let $X\to \PP^1_x$ be the Harbater-Katz-Gabber cover associated to the extension $\hat L/k((x))$. Note that $X\to \PP^1_x$ is an $I$-cover totally ramified at $x=0$ and $x=\infty$ with inertia group $I$ at $x=0$ and tamely ramified at $x=\infty$. Let $l\ge 1$ be coprime to $|I|$. By \cite[Prop 3.3]{harbater-adding.branch.points} there exists a cover $Y\to X$ such that the composition $Y\to X\to \PP^1_x$ is a $\Z/l\Z\wr I$-Galois cover and the inertia subgroup $I_1$ of $Y\to \PP^1_x$ at some point above $x=0$ is isomorphic to $I$.
 Now in view of Lemma \ref{groupconstruction}, Lemma \ref{galoiscorr} and Lemma \ref{splittingramification} the subfield $L=k(Y)^{I_1}$ of $k(Y)$ is contained in $k((x))$ and the Galois closure of $L/k(x)$ is $k(Y)$. Also $k(Y)k((x))\supset k(X)k((x))=\hat L$. This is all one is required to show in view of Lemma \ref{red.to.fin.extn}.
\end{proof}

\section{Arithmetic case} \label{arithmetic-case}
Now we deal with the number field case. In this case, the residue field is not separably closed so strict henselization is used instead of henselization.

\begin{thm}
Let $K$ be a number field, $\cO$ the ring of integers, $p$ a prime of $\cO$ and $R=\cO_p$ the local ring at $p$. Let $R^h$ be the strict henselization of $R$ and $K^h$ be the field of fractions of $R^h$. Then the Galois closure of the algebraic extension $K^h/K$ is algebraically closed.  
\end{thm}

\begin{proof}
 As in the previous cases we use Lemma \ref{red.to.fin.extn} to reduce the question to a statement on finite extensions of $K$. Let $M$ be the the local field obtained by completion at $p$, i.e. $M=\QF(\hat R)$. Let $M^h$ be the fraction field of the maximal unramified extension of $\hat R$. Note that $K^h$ is the algebraic closure of $K$ in $M^h$. Let $L_1/M$ be a finite extension, it is enough to show that there exists a finite extension $L/K$ such that $L\subset M^h$ and $\tilde LM^h$ contains $L_1$ where $\tilde L$ is the Galois closure of $L/K$.
 
 By passing to the Galois closure of $L_1/M$ we may assume $L_1/M$ is a Galois extension with Galois group $I$. By local field theory there exists an extension $L_0/K$ of the same degree as $[L_1:M]$ such that $L_0M=L_1$. Let $\tilde L_0$ be the Galois closure of $L_0/K$ and $H=\Gal(\tilde L_0/K)$ . Note that $\tilde L_0M \supset L_1$ and let $J$ be the inertia subgroup of the $H$-extension $\tilde L_0/K$ at a prime $q_0$ lying above $p$. So $J=\Gal(\tilde L_0M/M)=\Gal(L_0M^h/M^h)$.
 
 Let $l\ge 1$ be coprime to $|H|$ and $G$ be the wreath product $\Z/l\Z \wr H=(\Z/l\Z)^H\rtimes H$. By Lemma \ref{groupconstruction} we know that $\cap_{g\in G} gHg^{-1}=\{e\}$. Hence for the subgroup $J\le H\le G$ we have $\cap_{g\in G} gJg^{-1}=\{e\}$. By a result of Shafarevich (\cite{sha}), there exists a $G$-Galois extension $\tilde L/K$ dominating $\tilde L_0/K$. Moreover its proof in \cite[Claim 2.2.5]{topicsinGT}, shows that the extension $\tilde L/\tilde L_0$ can be arranged to be unramified over all primes of $\tilde L_0$ lying above $p$. It follows that $\tilde LM^h\supset L_1M^h$ and the inertia subgroup of $\tilde L/K$ at some prime $q$ lying above the prime $q_0$ in $\tilde L_0$ is the subgroup $\Gal(\tilde L_0M^h/M^h)=J$. 
 
  \[ \xymatrix{
         & \tilde LM^h \\	
  M^h\ar[ur] &          &      &\tilde L,\tilde q \ar[ull] \\
         &  \tilde L_0M\ar[uu] & L,q\ar@{-->}[ull] \ar@{-->}[ru]^J \\
M\ar[uu]\ar[ur] &  L_1\ar[u] & \tilde L_0, q_0\ar[ul] \ar[u]\ar[uur]_{(\Z/l\Z)^{|H|}}\\
M\ar@{=}[u]\ar[ur]_I&     & L_0\ar[ul]\ar[u]\\       
    &  K,p\ar[ur]_{|I|}\ar[ul]\ar[uur]_H \ar@{-->}[uuur]
    }
\]

 Let $L={(\tilde L)}^J$. Then $L/K$ is unramified at $q$. Hence $L\subset M^h$. Since $\cap_{g\in G} gJg^{-1}=\{e\}$, the Galois closure of $L/K$ is $\tilde L$ by Lemma \ref{groupconstruction}. This completes the proof because $\tilde LM^h\supset L_1$.
\end{proof}

\section{Local to global in number field case}\label{HKG}

Let $K$ be a number field, $v$ be a finite place of $K$ and $K_v$ be the completion of $K$ at $v$. Let $\hat L/K_v$ be a Galois extension with Galois group $G$. By ramification theory, we know that $G$ is a solvable group. The question which we want to investigate in this section is the following. Does there exist a $G$-Galois extension $L/K$ branched only at $v$ and $v$ is non-split in $L$ such that the completion of $L$ at its unique place $w$ lying above $v$ is $\hat L$? This is an analogue of Harbater-Katz-Gabber-cover in the number field case. The first thing to notice is that we must assume $\hat L/K_v$ is totally ramified. This is because $K_v$ has an infinite maximal unramified extension and an extension $L/K$ corresponding to an unramified extension of $K_v$ will be unramified. But Hilbert class field of $K$ is a finite extension of $K$.

When the group $G$ is abelian, it follows from class field theory (see \cite[Lemma 2.1.6]{topicsinGT}) that this question has an affirmative answer.

\begin{pro}\label{basecase}
 Let $G$ be an abelian group and $\hat L/K_v$ be a totally ramified $G$-Galois extension then there exists an extension $L/K$ branched only at $v$ such that $v$ lifts to a unique place of $w$ of $L$ and $L_w=\hat L$.
\end{pro}

\begin{proof}
 By local class field theory there is a surjective map (Artin homomorphism) $\epsilon_v:K_v^*\to G$ with kernel $N_{\hat L/K_v}(\hat L^*)$. For other places $u$ of $K$ define $\epsilon_u: K_u^*\to G$ to be trivial. By global class field theory, there exists a $G$-Galois extension $L/K$ whose local Artin homomorphisms agree with $\epsilon_u$ on the inertia subgroup at $u$ for a place $u$ of $K$. Since $\epsilon_u$ is trivial for all places $u$ of $K$ different from $v$, $L/K$ is branched only at $v$. Note that $\hat L$ defines a place $w$ of $L$ lying above $v$. Since $\Gal(L/K)=\Gal(L_w/K_v)$, we obtain that $v$ is non-split in $L$. 
\end{proof}

The above result can be extended for arbitrary $|G|$ in certain situations.

\begin{pro}
  Let $G$ be a group and $\hat L/K_v$ be a totally ramified $G$-Galois extension and $G=G_{-1} = G_0\ge G_1\ge G_2\ldots G_n=\{e\}$ be the ramification filtration on $G$ then for $1\le i\le n$ there exist extensions $L_i/K$ branched only at $v$ such that $v$ lifts to a unique place of $w_i$ of $L_i$, $L_i\subset L_{i+1}$ is a $G_i/G_{i+1}$-Galois extension and $L_{w_i}=\hat L^{G_i}$.
\end{pro}

\begin{proof}
 First we note that $G$ is a solvable group. In the ramification filtration on $G$, $G=G_{-1} = G_0\ge G_1\ge G_2\ldots G_n=\{e\}$ is such that $G_0=G$ is the inertia group, $G_1$ is a $p$-group and $G_0/G_1$ is a cyclic prime-to-$p$ group where $p$ is the characteristic of the residue field. Also $G_i/G_{i+1}$ for $1\le i\le n-1$ are elementary abelian $p$-groups. 
 
 All fields are viewed as subfields of a fixed algebraic closure of $K_v$.
 We will construct extensions $L_i/K$ inductively for $G/G_i$-Galois subextensions $\hat L_i/K_v$ where $\hat L_i=L^{G_i}$ satisfying the conclusion of the theorem. For $i=1$, we can use Proposition \ref{basecase} to obtain $L_1$. Now assume that we have constructed the extension $L_i/K$ branched only at $v$, $w_i$ is the unique place of $L_i$ lying above $v$ and the completion of $L_i$ at $w_i$ is $\hat L_i$.
 
 Again applying Proposition \ref{basecase} to the $(G_i/G_{i+1})$-Galois extension $\hat L_{i+1}/\hat L_i$ we obtain a $(G_i/G_{i+1})$-Galois extension $L_{i+1}/L_i$ branched only at $w_i$ such that $w_i$ lifts to the unique place $w_{i+1}$ of $L_{i+1}$ and the completion of $L_{i+1}$ at $w_{i+1}$ is $\hat L_{i+1}$. So the place $v$ of $K$ lifts to the unique place $w_{i+1}$ of $L_{i+1}$ and $L_{i+1}/K$ is branched only at $v$.
\end{proof}

\begin{thm}
 Under the notations and assumptions of the above result, if we further assume that all the prime factors of the class number of $L$ is greater than $|G|$. Then $L/K$ is a Galois extension with Galois group $G$.
\end{thm}

\begin{proof}
 Since $L/K$ is totally ramified at $v$, for any intermediate field $M$, $L/M$ is also totally ramified at the place of $M$ lying above $v$. Let $M'$ be the Hilbert class field of $M$. Then the extension $M'/M$ is unramified and hence $L$ and $M'$ are linearly disjoint over $M$. It follows that the field $LM'$ is an unramified extension of $L$ and $[LM':L]=[M':M]$. But the primes dividing class number of $L$ are greater than $|G|$. Hence so is primes dividing $[LM':L]$. So all primes dividing class number of $M$ is also greater than $|G|$.
 
 We already have that $L_1/K$ is a $G/G_1$-Galois extension. Assume that $L_i/K$ is a $G/G_i$-Galois extension, we shall show that $L_{i+1}/K$ is a Galois extension. 
 
 Let $\tilde L_{i+1}$ be the Galois closure of $L_{i+1}/K$. Note that $[\tilde L_{i+1}:K]$ is a factor of $|G|!$, since $[L_{i+1}:K]=|G|$. Let $\beta \in L_{i+1}$ be such that $L_{i+1}=K(\beta)$. Note that $\hat L_{i+1}=K_vL_{i+1}=K_v(\beta)$. Let $f(x)$ be the minimal polynomial of $\beta$ over $K$. Note that $\tilde L_{i+1}$ is the splitting field of $f(x)$. Since $[L_{i+1}:K]=[\hat L_{i+1}:K_v]$, $f(x)$ is also the minimal polynomial of $\beta$ over $K_v$. Also $\hat L_{i+1}/K_v$ is Galois, hence all the roots of $f(x)$ are in $\hat L_{i+1}$. So we have $\tilde L_{i+1} \subset \hat L_{i+1}$. Let $\tilde w_{i+1}$ be the place of $\tilde L_{i+1}$ corresponding to the DVR $\hat \cO_{L_{i+1}}\cap \tilde L_{i+1}$ and note that $\tilde w_{i+1}$ is lying above $w_{i+1}$. Also the decomposition subgroup at $\tilde w_{i+1}$ of $\tilde L_{i+1}/K$ is $\Gal(\hat L_{i+1}/K_v)$ is same as the decomposition group of $L_{i+1}/K$ at $w_{i+1}$. Hence $\tilde L_{i+1}/L_{i+1}$ is an unramified extension. But by the first paragraph, primes dividing class number of $L_{i+1}$ are greater than $|G|$. Where as $[\tilde L_{i+1}:L_{i+1}]$ divides $[\tilde L_{i+1}:K]$ which in turn is a factor of $|G|!$. So $L_{i+1}= \tilde L_{i+1}$ is a Galois extension of $K$. The Galois group is $G/G_{i+1}$ follows from the fact that $\hat L_{i+1}=L_{i+1}K_v$, $\Gal(\hat L_{i+1}/K_v)=G/G_{i+1}$ and $[L_{i+1}:K]=[\hat L_{i+1}:K_v]$. The result follows by induction.
\end{proof}

\begin{rmk}
These global extensions corresponding to local extensions are useful because they can be used to construct extensions of number fields with some control over even wild ramification. This can be achieved along the lines of \cite{wild.ram}. 
\end{rmk}

\end{document}